\documentclass[12pt]{amsart}
\usepackage{amsthm,amssymb}

\newtheorem{thm}{Theorem}[section]
\newtheorem{lm}[thm]{Lemma}
\newtheorem{cor}[thm]{Corollary}
\newtheorem{pro}[thm]{Proposition}

\theoremstyle{definition}
\newtheorem{df}[thm]{Definition}

\input epsf

\begin{document}

\title[On the semi-threading of knot diagrams with minimal overpasses]
{On the semi-threading of knot diagrams with minimal overpasses}

\author{Jae-Wook Chung}
\address{Department of Mathematics, Yeungnam University,
Kyongsan, Korea 712-749}
\email{jwchung@ynu.ac.kr}

\author{Seulgi Jeong}
\address{Department of Mathematics \\Kyonggi University
\\ Suwon, 443-760 Korea}
\email{seul@kyonggi.ac.kr}

\author{Dongseok Kim}
\address{Department of Mathematics \\Kyonggi University
\\ Suwon, 443-760 Korea}
\thanks{Corresponding author was supported by the Korea Research Foundation Grant funded by the Korean Government (MOEHRD, Basic Research Promotion Fund) (KRF-2008-331-C00035).}

\begin{abstract}{Given a knot diagram $D$, we construct a
semi-threading circle for it which can be an axis of $D$ as a
closed braid depending on knot diagrams. In particular, we
consider semi-threading circles for minimal diagrams of a knot
with respect to overpasses which give us some information related
to the braid index. By this notion, we show that, for every
nontrivial knot $K$, the braid index $b(K)$ of $K$ is not less
than the minimum number $l(K)$ of overpasses of diagrams. Moreover,
they are the same for a torus knot.}
\end{abstract}

\maketitle

\section{Introduction}

Throughout this paper, every knot is oriented and lies in the
3-dimensional sphere $S^3$. Also, all knots are isotopic
to polygonal or smooth knots, i.e., they are \emph{tame}. Therefore,
each knot has a diagram with finitely many crossings, hence,
has a finite number of overpasses. A diagram $D$ of a knot $K$
is a popular representative of the isotopy class of $K$ which is
also called the \emph{knot type} of $K$. $D$ is obtained from a regular
projection $P$ of $K$ by the following steps. First, we take a
sufficiently small neighborhood of each double point of $P$ so
that the intersection of the neighborhood and $P$ looks like an
`X-shape' on the plane. Second, we adjust the interior of each
neighborhood so that we have a knot $D$ which is isotopic to $K$
and regularly projected to $P$. In this sense, a knot diagram is
an `almost planar' knot, i.e., it lies in the plane except for a
small neighborhood of each double point of the regular projection.
It has been found various diagrams of knot types whose numbers
of crossings are minimal. Here, we prefer the number of overpasses
of a knot diagram and are interested in knot diagrams whose numbers
of overpasses are minimal. This gives us another point of view to
consider knot diagrams. Notice that the minimal number of overpasses
of a knot is the classical \emph{bridge number} and first studied
by Schubert in \cite{SCH}, where the effect of various operations
on knots (satellite, cabling, connected sum) on this number was
investigated. We refer to \cite{A} for further studies. The first author investigated the relationships
between the number of crossings and the number of overpasses of
a knot diagram. It turns out the number of overpasses is estimated
by that of crossings if the knot diagram has a minimal number of
crossings \cite{C-L}. On the other hand, we define a `semi-threading'
of an oriented knot diagram and mainly consider `minimal semi-threading'
based on our definition. From this construction, we know that the
braid index of a nontrivial knot is not less than the minimum number
of overpasses of its diagrams. This theorem suggests important
information about minimal knot with respect to overpasses.
In particular, the braid index of each torus knot is the same as
the minimum number of overpasses of its diagrams. That is,
the semi-threading circle is just a threading. For a `threading'
of a knot in detail, see \cite{MOR}.

\section{Minimal crossings and minimal overpasses of knot type}

Given a knot diagram $D$, the number of crossings or the crossing
number of $D$ is denoted by $c(D)$. For each knot $K$, we denote
$min\{c(D) \mid D$ \textit{is a diagram of} $K\}$ by $c(K)$. Note
that we may assume a knot diagram $D$ lies in the plane by
indicating `overcrossings' and `undercrossings'. A `crossing',
in this sense, of a knot diagram $D$ means a `signed double point'
of the regular projection of $D$. Hence, $c(D)$ is the number of
all double points of the regular projection of $D$. On the other
hand, we may regard a crossing of $D$ as the pair of two points
overcrossing and undercrossing in $D$ which are projected to the
same double point. That is, a crossing is considered as the pre-image
of a double point under the projection map.

\begin{pro} Let $D$ be a knot diagram. Then there are a unique
nonnegative integer $k$ and a finite sequence $s_1,f_1,s_2,f_2,
\dots,s_k,f_k$ of $2k$ points of $D$ which are neither overcrossings
nor undercrossings of $D$ such that
$$[s_1,f_1],[s_2,f_2],\dots,[s_{k-1}, f_{k-1}],[s_k,f_k]$$ and
$$[f_1,s_2],[f_2,s_3],\dots,[f_{k-1},s_k],[f_k,s_1]$$ are the
overpasses and the underpasses of $D$ with respect to the
sequence, respectively, where $[s_i, f_i]$, for each $i \in
\{1,\dots,k\}$, is the closed arc of $D$ from $s_i$ to $f_i$ which
contains at least one overcrossing but has no undercrossing;
similarly, $[f_i,s_{i+1}]$, for each $i \in \{1,\dots,k-1\}$, is the
closed arc of $D$ from $f_i$ to $s_{i+1}$ which contains at least
one undercrossing but has no overcrossing; also, the last one
$[f_k, s_1]$ is the closed arc of $D$ from $f_k$ to $s_1$ which
contains at least one undercrossing but has no overcrossing.
\end{pro}

In Proposition 2.1, such a sequence $s_1,f_1,s_2,f_2,\dots,s_k,f_k$
is called an \textit{over-underpass sequence} of $D$. Since any
over-underpass sequence of $D$ consists of $2k$ points, the number
of overpasses (or underpasses) with respect to any over-underpass
sequence of $D$ is $k$. Hence, we can define the number of overpasses
(or underpasses) of the knot diagram $D$ as $k$, and denote it by
$l(D)$, called the length of over-underpass sequences or the bridge
number of $D$. Also, for each knot $K$, we denote
$min\{l(D) \mid D$ \textit{is a diagram of} $K\}$ by $l(K)$.

Notice that (1) $c(D)$ and $l(D)$ are plane isotopy
invariants of knot diagrams, i.e., if $D_1$ and $D_2$ are plane
isotopic knot diagrams, then $c(D_1)=c(D_2)$ and
$l(D_1)=l(D_2)$; (2) $c(K)$ and $l(K)$ are isotopy
invariants of knots, i.e., if $K_1$ and $K_2$ are isotopic knots,
then $c(K_1)=c(K_2)$ and $l(K_1)=l(K_2)$.

\begin{cor} If $l(D)\leq 1$, then $D$ is a diagram of a
trivial knot. Therefore, a knot $K$ is trivial if and only if $K$
has a diagram $D$ with $l(D)\leq 1$.
\end{cor}

Obviously, the converse is not true. A diagram of a double twisted
circle can be an example. Similarly, for any positive integer $k$,
there is a diagram of a trivial knot whose number of overpasses is
greater than $k$. On the other hand, given a knot diagram $D$ with
at least one crossing, we can add crossings to $D$ as many as we
want without changing the knot type and the number of overpasses
of $D$. Take a sufficiently small arc of $D$ from $s_1$ to a point
between $s_1$ and the first overcrossing of $D$ from $s_1$ and
twist it alternatingly so that the number of overpasses of $D$ is
not changed. Or, we may modify the interior of a sufficiently
small neighborhood of a crossing of $D$. Hence, we have the
following corollary.

\begin{cor} If $D$ is a diagram of a knot $K$ such that
$c(D)\geq 1$, then for every positive integer $n$, there is a
diagram $D'$ of $K$ such that $l(D')=l(D)$ and
$c(D') \geq c(D)+n$.
\end{cor}

\begin{lm} $l(D)\leq c(D)$ for any knot diagram $D$. The
equality holds if $D$ is an alternating knot diagram. Furthermore,
$l(K)\leq c(K)$ for any knot $K$.
\end{lm}

Remark that, since the number of double points of $D$ is finite
and we have an $\epsilon$-neighborhood of $D$ for a sufficiently
small $\epsilon$ which is the regular projection of a knotted
solid torus whose axis is the knot diagram $D$, we can always take
such an $\epsilon$ and such $\epsilon$-neighborhoods as above.

The following theorem gives us a relationship between the numbers of
crossings and overpasses.

\begin{thm}\cite{C-L}. If $D$ is a minimal diagram of a nontrivial
knot $K$ with respect to crossings, then
$1+\sqrt{1+c(D)}\leq l(D)\leq c(D)$.
\end{thm}

Notice that a knot $K$ need not hold $1+\sqrt{1+c(D)}\leq l(D)$.
The closure of the braid $b_1^5$, where $b_1$ is the generator of
the standard group presentation of the braid group $B_2$, can be
shown as an example for it by the Theorem 3.6 in the next section.
Also, remark that the crossing number of a knot diagram with a
minimal number of overpasses can be arbitrarily large as shown at
Corollary 2.3.

\section{Semi-threading knot diagram and braid index}

In this section, we briefly introduce a definition of threading
of a knot diagram. For details, see \cite{MOR}. Also, we define a
`semi-threading' of a knot diagram. In particular, we consider the
semi-threading of minimal diagrams of a knot with respect to overpasses.

First, let us introduce Alexander braiding theorem and Markov theorem
for braids shortly.

(1) Alexander Braiding Theorem: Every link is the closure of a braid,
that is, a closed braid.

This theorem is published by J.W. Alexander in 1923 and allows one to
study knots and links using the theory of braids \cite{Ale}. The minimum
number of braid strands used in the closure is called the
\textit{braid index} of the link. Note that the closures of different
braids need not be distinct. An answer for this question is the
following theorem.

(2) Markov Theorem: The closures of two braids are isotopic if and
only if one braid can be obtained from the other by a finite sequence
of \textit{Markov moves}, which are sometimes called conjugations and
stabilizations, respectively \cite{Mk}.

\begin{df} \cite{MOR}. Let $K$ be a knot, $T$ a trivial knot, and
$K \sqcup T$ an oriented link in $S^3$ whose components are only
$K$ and $T$ and the linking number $lk(K,T)$ is positive.
The link $K \sqcup T$ is said to be \textit{braided} if there are
isotopic diagrams $K'$, $T'$, $K' \sqcup T'$ of $K$, $T$,
$K \sqcup T$, respectively, such that $K'$ is a closed braid, i.e,
a closure of a braid, and $T'$ is an axis of $K'$, which is called
a \textit{threading circle} of $K'$. The link diagram $K' \sqcup T'$
is called a \textit{threading} of $K'$.
\end{df}

By the definition of a threading circle of a knot diagram, we easily
know that, for each knot $K$, the braid index $b(K)$ is
the same as $min\{lk(D,L) \mid D$ \textit{is a diagram of K and L is a
threading circle of D}$\}$. Notice that $b(K)$ is an isotopy invariant
of knots.

Actually, a threading circle of a knot diagram is an axis of the knot
as a closed braid. We define a `weaker threading' as follows.

\begin{df} Let $D$ be a diagram of a knot. An oriented simple
closed curve $L$ on the plane is called a \textit{semi-threading
circle} of $D$ if $L$ crosses all overpasses of $D$ transversely
exactly once so that the linking number $lk(D,L)$ of $D$ and $L$
is positive. We call such a link $D \sqcup L$
a \textit{semi-threading} of $D$.
\end{df}

Now we construct a semi-threading circle for a knot diagram which
might be an axis of $D$ as a closed braid depending on knot diagrams.
However, a threading circle of a knot diagram need not be a
semi-threading circle of it because a semi-threading circle requires
all the overpasses.

\begin{thm} A semi-threading circle $L$ of a knot diagram $D$ exists.
\end{thm}

\begin{proof} If $D$ is a diagram of a trivial knot with no crossing,
we can draw an oriented simple closed curve $L$ on the plane so
that $D \sqcup L$ is a diagram of a positive Hopf link, a two
component link whose components are trivial and linking number is
1. Obviously, $L$ is a semi-threading circle of $D$.

Suppose $l(D)=k \geq 2$ and $s_1,f_1,s_2,f_2,\dots,s_k,f_k$ is an
over-underpass sequence of $D$. By a plane isotopy, we can arrange
all overpasses $[s_1,f_1],[s_2,f_2],\dots,[s_k,f_k]$ to be vertical
downward and from left to right, i.e., we have vertical overpasses
$[s_1,f_1],[s_2,f_2],\dots,[s_k,f_k]$ from left to right.

Now let us construct a simple closed curve $L$ on the plane to be
a semi-threading circle of $D$.
For each $i \in \{1,\dots,k\}$,
choose a point $x_i$ on $(s_i,f_i)$ which is not a crossing of $D$,
where $(s_i,f_i)=[s_i,f_i]-\{s_i,f_i\}$. Take a sufficiently small
circular closed neighborhood of $x_i$ in the plane so that the
intersection of the neighborhood and $D$ is a closed arc contained
in $(s_i,f_i)$ which has no crossing of $D$. Choose two points
$x_i'$ and $x_i''$ on the boundary of the neighborhood so that the
line segment $[x_i',x_i'']$ is perpendicular to $[s_i,f_i]$ and
passes through $x_i$ from left to right.
Draw the straight line $l_i$ from $x_i''$ to $x_{i+1}'$ for each
$i \in \{1,\dots,k-1\}$. If $l_i$ tangents $D$, by a sufficiently
small change, we can modify $l_i$ to a piecewise straight line which
intersects $D$ transversely and does not intersect the overpasses.
We also denote the piecewise straight line by $l_i$.
By the above construction, we have a piecewise straight line
$$L_0=[x_1',x_1''] \cup l_1 \cup \dots \cup [x_{k-1}',x_{k-1}''] \cup
l_{k-1} \cup [x_k',x_k'']$$ from $x_1'$ to $x_k''$ which intersects
$D$ transversely and intersects each overpass exactly once.
Let $\epsilon$ be a sufficiently small
positive real number such that

(1) $L_0+\epsilon$ is the parallel transition of $L_0$ by $\epsilon$
along the overpasses;

(2) $L_0+\epsilon$ does not pass through any crossings of $D$;

(3) $L_0+\epsilon$ intersects $D$ transversely;

(4) there is no crossing of $D$ between $L_0$ and $L_0+\epsilon$;

(5) each of the vertical line segments $[x_1',x_1'+\epsilon]$ and
$[x_k'',x_k''+\epsilon]$ does not intersect $D$. Now we get a
piecewise straight line $l_k'$ from $x_k''$ to $x_1'$ as
$$l_k'=[x_1',x_1'+\epsilon] \cup (L_0+\epsilon) \cup
[x_k'',x_k''+\epsilon]$$ and a simple closed curve $L_0
\cup l_k'$ on the plane.
Let $L=L_0 \cup l_k'$ and give the orientation which agrees with
from $x_1'$ to $x_k''$ clockwise.

As the next step, we give the crossing structures to the intersections
of $D$ and $L$ to make the linking number $lk(D,L)$ positive. For each
$i \in \{1,\dots,k\}$, $[x_i',x_i'']$ crosses $[s_i,f_i]$ below, and,
for each $i \in \{1,\dots,k-1\}$, $l_i$ crosses above the underpasses of $D$
which intersect $l_i$. In particular, $l_k'$ crosses above at any
intersection with $D$. Therefore, $L$ is a semi-threading circle of $D$
and the link diagram $D \sqcup L$ is a semi-threading of $D$.
\end{proof}

In the proof of Theorem 3.3, such a trivial knot diagram $L$ is unique
up to isotopy and called a \textit{minimal semi-threading circle} of $D$.
Also, we denote it by $M_D$. The semi-threading $D \sqcup M_D$ of $D$ is
called a \textit{minimal semi-threading} of $D$.
Remark that we can construct a semi-threading circle without changing
the knot diagram. In other words, the plane isotopy to make all overpasses
vertical downward is not necessary. It is obvious from the construction
of the semi-threading circle in the above proof.

We have already had well-known threadings devised by
Morton \cite{MOR}. They are also very good examples of threadings.
As an advantage of minimal semi-threading, it immediately shows the
relationship between the number of overpasses and the linking
number of threading, Corollary 3.4.
On the other hand, Morton's threading depends on unknotted simple closed
curves and numbers of crossings between knot and unknot diagrams, but
in his paper \cite{MOR}, Morton proved valuable theorems on his own threading.

In general, a threading circle of a knot diagram need not require all
overpasses of it. However, each minimal semi-threading of $D \sqcup M_D$
and the threading $D \sqcup Morton's$ for a knot diagram $D$ still require
all overpasses of $D$. For our purpose on this paper, the notion of minimal
semi-threading is necessary as a key fact. The following statement is an
immediate consequence of the minimal semi-threading of a knot diagram.

\begin{cor} If $D$ is a knot diagram with $l(D) \geq 1$, then
$l(D)=lk(D,M_D)$.
\end{cor}

The following lemma plays an important role between the minimal number
of overpasses and the braid index of a knot type. It may be regarded as
another version of Alexander Braiding Theorem.

\begin{lm} If $D$ is a knot diagram and a minimal semi-threading
$D \sqcup M_D$ is not braided, then there is a threading circle
$T$ of $D$ such that $lk(D,M_D) < lk(D,T)$.
\end{lm}

\begin{proof} Suppose that $D$ is a knot diagram and $l(D)=k \geq 2$
and $s_1,f_1,s_2,f_2,\dots,s_k,f_k$ is an over-underpass sequence of $D$.
By Corollary 3.4, $lk(D,M_D)=k$.
Let us use $L_0$ in the proof of Theorem 3.3. That is,
$$L_0=[x_1',x_1''] \cup l_1 \cup \dots \cup [x_{k-1}',x_{k-1}''] \cup
l_{k-1} \cup [x_k',x_k''].$$ We draw rays $l_{L_0}$ and $r_{L_0}$ starting
from $x_1'$ and $x_k''$ leftward and rightward which extend the line segments
$\overline{x_1',x_1''}$ and $\overline{x_k',x_k''}$, respectively, and
cross above the underpasses of $D$ which intersect them.
Let $T_0=l_{L_0} \cup L_0 \cup r_{L_0}$. We can think of $T_0$ as
a simple closed curve passing through the infinity $\infty$ of $S^3$.
Hence, $D \sqcup T_0$ is isotopic to $D \sqcup M_D$. We call $T_0$ an
extended minimal semi-threading circle of $D$.

Our goal is to modify the extended minimal semi-threading circle $T_0$ to a
threading circle $T$ so that $D \sqcup T$ is braided. Note that, for each
$i \in \{1,\dots,k\}$, the $i$-th underpass $[f_i,s_{i+1}]$ or $[f_k,s_1]$
crosses below $T_0$ an odd times.
Notice that the extended semi-threading $D \sqcup T_0$ is really a
threading, that is, $D \sqcup T_0$ is braided if each underpass of $D$
crosses below $T_0$ exactly one time.

By hypothesis, $D \sqcup T_0$ is not braided, so $D$ has at least one
underpass which crosses $T_0$ below more than one time.

Suppose that $1 \leq m \leq k$ and $\{u_1,\dots,u_m\}$ is the set of
all underpasses of $D$ each of which crosses $T_0$ below more than one time.
For each $i \in \{1,\dots,m\}$, there is $n_i \in \mathbb N$ such that
$u_i$ crosses $T_0$ below exactly $2n_i+1$ times.
Give the order to $2n_i+1$ undercrossings on $u_i$ by $T_0$ which agree with
the orientation of $D$. We do crossing change for all the even numbered
undercrossings. That is, $n_i$ times of crossing change occur. Hence, exactly
$n_1+\cdots+n_m$ times of crossing change occur on $T_0$.
Let $T$ be this modified $T_0$. Then $D \sqcup T$ is braided with
$lk(D,T)=lk(D,M_D)+n_1+\cdots+n_m$. This proves the lemma.
\end{proof}

Notice that, in the proof of Lemma 3.4, $D \sqcup M_D$ and $D \sqcup T$
are surely not isotopic. Even though the threading $D \sqcup T$ is a
diagram of a braided link, it may have lots of unnecessary strands
as a closed braid. These unnecessary strands can be reduced by a suitable
sequence of Markov moves.

Now we show an inequality between the number of minimal overpasses and
the braid index for a nontrivial knot.
From now on, we consider only minimal knot diagrams with respect to overpasses.

\begin{thm} If $K$ is a nontrivial knot, then $l(K) \leq b(K)$.
\end{thm}

\begin{proof} Let $K$ be a nontrivial knot, and let $D$ be a minimal diagram
of $K$ with respect to overpasses. Hence, $l(D)=l(K)$. Consider an
extended minimal semi-threading $D \sqcup T_0$ of $D$. If $D \sqcup T_0$ is
braided, then $b(K)=lk(D,T_0)$, so $l(K) \leq b(K)$. Suppose that $D \sqcup T_0$
is not braided. Using plane isotopy, we remove all unnecessary crossings
between the underpasses of $D$ and the extended minimal semi-threading circle
$T_0$ keeping on the structure of over-underpass sequence of $D$.
Assume that $D'$ is the modified diagram of $D$. By Corollary 3.4 and Lemma 3.5,
$l(K)=lk(D',T_0) \leq lk(D',T)$, where $T$ is such a threading circle of $D'$
modified from $T_0$ as in the proof of Lemma 3.5. Since all unnecessary crossings
between the underpasses and $T_0$ are removed, the threading $D' \sqcup T$ has
no unnecessary strands as closed braid. Therefore, $b(K)=lk(D',T)$. This proves
the theorem.
\end{proof}

In general, for a knot $K$, it is not true that $l(K)=b(K)$.
As an example, the figure 8-knot $K$ has $l(K)=2$ but $b(K)=3$.
Our main concern is what conditions for a knot $K$ make $l(K)=b(K)$.
One of such special classes of knot types is the torus knots.

\begin{thm} If $K$ is a torus knot, then $l(K)=b(K)$.
\end{thm}

\begin{proof} Suppose that $K$ is a $(p,q)$-torus knot, where
$p$ and $q$ are integers which are relatively prime.
Let $D$ be a standard diagram of $K$. Without loss of generality,
we may ssume that $0<p<q$. Then $c(D)=pq-q$ and $l(D)=q$.
In this case, $c(K)=c(D)$ by \cite{MUR} and $b(K)=b(D)=p$. By Theorem 3.6,
$l(K) \leq b(K)$. Since $K$ is a $(p,q)$-torus knot, $K$ is isotopic to
a $(q,p)$-torus knot $K'$. Let $D'$ be a standard diagram of $K'$. Then
$c(D')=pq-p$ and $l(D')=p$. Since the minimal number of crossings or overpasses
of knot diagrams is an isotopy invariant of knots, $c(K')=c(K)=c(D)<c(D')$
but $l(K)=l(K')=l(D')=p$. Therefore, $b(K)=l(K)=p$.
\end{proof}

This theorem says that a knot diagram with a minimal number of
overpasses gives a piece of information of the braid index of
the knot type. In other words, given a diagram $D$ of a knot $K$,
try to remove all unnecessary overpasses of $D$. When we can do so,
we will approach the braid index. In general, it is very hard to find
$l(K)$ and $b(K)$ for a knot $K$. Hence, it is very valuable that
we know more exact relationship between $l(K)$ and $b(K)$.

\section{Representing threading knot diagram as a closed braid}

In this section, we explain how to get an isotopic closed braid from
our threading of a knot diagram.

Consider a threading $D \sqcup T$ of a knot diagram $D$ with a threading
circle $T$. Let $s_1,f_1,s_2,f_2,\dots,s_k,f_k$ be an over-underpass sequence
of $D$. We may assume all overpasses of $D$ lie in
$R_+^3=\{(x,y,z)\in R^3 \mid z > 0\}$ and all underpasses of $D$ are
on the plane, i.e., the $xy$-plane of $R^3$.

By an isotopy, change the $i$-th overpass $[s_i,f_i]$ for
each $i \in \{1,\dots,k\}$ to the semi-circle $\widehat{s_i,f_i}$ in
$R_+^3$ from $s_i$ to $f_i$ whose projection is that of
$[s_i,f_i]$ and $T$ to a straight line $T'$ on the plane which
intersects perpendicularly to the projection of each semi-circle,
i.e., we can think of $T'$ as a simple closed curve passing through
the infinity $\infty$ of $S^3$.

Next, we modify the underpasses of $D$ to get the desired knot.
Now the plane contains only $T'$ and the underpasses of $D$.
Take a positive real number $\alpha$ and push down all underpasses of
$D$ by $\alpha$ so that they are on the plane $z=-\alpha$.

For each $i \in \{1,\dots,k-1\}$, let $u_i=[f_i,s_{i+1}]-\alpha$ and
$u_k=[f_k,s_1]-\alpha$, where $u_i$ is the parallel transition of the
$i$-th underpass by $-\alpha$,
and, for each $i \in \{1,\dots,k\}$, let $l_{s_i}$ and $l_{f_i}$ be
the vertical line segments from $s_i-\alpha$ to $s_i$ and from
$f_i$ to $f_i-\alpha$, respectively. Then $$(u_1 \cup u_2 \cup \dots
\cup u_k) \cup (l_{s_1} \cup l_{s_2} \cup \dots \cup l_{s_k})$$
represents a braid. To show this, fix the points $f_1-\alpha$, \dots,
$f_k-\alpha$ and lift up $l_{s_1}$, \dots, $l_{s_k}$. Then we can get the braid.
Now let $K'$ be $$(\widehat{s_i,f_i} \cup \dots \cup \widehat{s_k,f_k}) \cup
(l_{f_1} \cup \dots \cup l_{f_k}) \cup (u_1 \cup \dots \cup u_k) \cup
(l_{s_1} \cup \dots \cup l_{s_k}).$$ Then $D$ represents $K'$ and
$D \sqcup T$ also represents $K' \sqcup T'$ as desired.


\bigskip
\medskip


\bigskip


\bigskip

\begin{thebibliography}{BM}

\bibitem{A} C. Adams, The Knot Book, W. H. Freeman \& Co., New York, 1994.
\bibitem{MOR} H. R. Morton, \textit{Threading knot diagrams.}, Math. Proc. Camb. Phil. Soc.
\textbf{99} (1986), 247--260.
\bibitem{C-L} J.-W. Chung and X.-S. Lin, \textit{On the bridge number of knot diagrams
with minimal crossings.}, Math. Proc. Camb. Phil. Soc.
\textbf{137} (2004), 617--632.
\bibitem{Ale} J. W. Alexander, \textit{A lemma on systems of knotted curves.}, Proc. Nat. Acad. Sci. U.S.A.
\textbf{9} (1923), 93--95.
\bibitem{Mk} A. A. Markov, \textit{Uber die freie Aquivalenz geschlossener Zopfe}, Recueil
Mathematique Moscou, \textbf{1} 1935, 73--78
\bibitem{MUR} K. Murasugi, \textit{On the braid index of alternating links},
Trans. Amer. Math. Soc. \textbf{326} (1991), 237--260
\bibitem{SCH} H. Schubert, \textit{\"Uber eine numerische Knoteninvariante},
Math. Z. \textbf{61} (1954), 245--288.

\end{thebibliography}
\end{document}